\documentclass[10pt,reqno]{amsart}
\usepackage{graphicx,stmaryrd}
\usepackage{indentfirst,csquotes}
\usepackage{etoolbox}
\appto\normalsize{\belowdisplayshortskip=\belowdisplayskip}

\usepackage{amssymb,amsthm,amsmath, enumerate}
\usepackage{xcolor,paralist,hyperref,titlesec,fancyhdr,etoolbox}
\newtheorem*{theorem-non}{Theorem}
\usepackage[natbibapa]{apacite}
\newtheorem{theorem}{Theorem}[]

\newtheorem{lemma}{Lemma}
\newtheorem{proposition}{Proposition}
\newtheorem{corollary}{Corollary}

\newtheorem{assumption}{Assumption}

\titleformat{name=\section}{}{\thetitle.}{0.8em}{\centering\scshape}
\titlespacing*{\section}{0pt}{0ex}{0ex}

\hypersetup{ colorlinks=true, linkcolor=black, filecolor=black, urlcolor=black }

\usepackage{lipsum}
\begin{document}
\title[On the Vector Extension of the Bellman Equations]{A Counterexample and a Corrective to the Vector Extension of the Bellman Equations of a Markov Decision Process} 

\author[Anas Mifrani]{Anas Mifrani*\vspace{-1em}}
\maketitle

\let\thefootnote\relax
\footnotetext{*: Toulouse Mathematics Institute, University of Toulouse, 118 Rte de Narbonne, Toulouse, France} 
\footnotetext{Email address: anas.mifrani@math.univ-toulouse.fr} 
\footnotetext{ORCID: 0009-0005-1373-9028} 

\begin{abstract}
Under the expected total reward criterion, the optimal value of a finite-horizon Markov decision process can be determined by solving the Bellman equations. The equations were extended by D. J. White to processes with vector rewards in 1982. Using a counterexample, we show that the assumptions underlying this extension fail to guarantee its validity. Analysis of the counterexample leads us to articulate a sufficient condition for White's functional equations to be valid. The condition is shown to be true when the policy space has been refined to include a special class of non-Markovian policies, or when the dynamics of the model are deterministic, or when the decision making horizon does not exceed three time steps. The paper demonstrates that, in general, the solutions to White's equations are sets of Pareto efficient policy returns over the refined policy space. Our results are illustrated with an example.
\end{abstract} 

\noindent%

\section{Introduction}
In a seminal paper on vector-valued Markov decision processes, Douglas J. White \citep[\textit{J. Math. Anal. Appl. 89(2)}]{white1982} presents an inductive scheme for determining the Pareto efficient set of policy returns from any period $t$ onward given any initial state $s_t$. The scheme is intended as a generalization of the value iteration algorithm \citep{puterman1994} for scalar Markov decision processes, and is based accordingly on equations reminiscent of Bellman's, though the unknowns in White's case are set-valued rather than real-valued functions. There is abundant reference to White's equations in the technical literature \citep{3, 4, 5, 6, 7, 8}, and their continuing relevance is evidenced by the recent publication of a paper \citep{9} which uses them explicitly as the basis for a novel multi-objective dynamic programming algorithm. Unfortunately, however, the equations are invalid under White's own assumptions. By means of a counterexample, we shall demonstrate that their solution does \textit{not} yield the desired efficient sets, contrary to White's chief claim (see Theorem 2 \citep{white1982}, reproduced below). We shall also see that the proof of this ``theorem" \textsuperscript{1}\footnote{\textsuperscript{1} Throughout this paper, the use of scare quotes around the word ``theorem" indicates that the word is not accurate, since theorems are true statements by definition.} \citep[p. 7]{white1982} relies critically on a rather intuitive, but in fact erroneous, argument.

It is also the purpose of this paper to develop conditions in which the equations are applicable. In particular, it shall be shown that the equations hold when the decision making horizon spans two or three time steps (Section 3), or when the dynamics of the model are deterministic (Section 4), or when the policy space is augmented with a class of non-Markovian policies (Section 5). By analyzing the latter case, we discover that the actual solutions to White's equations are Pareto efficient sets of policy returns over the augmented policy space (Theorem 5, Sec. 5). To our knowledge, this paper carries the first such analysis, and is the first to disprove the said ``theorem".

To make for a relatively self-contained paper, we shall restate White's assumptions \citep{white1982}, changing the notation slightly. Let \(S\) be a finite set of states. For each state \(s \in S\), let \(A_s\) be the set of actions that can be taken in state \(s\); suppose \(A_s\) is compact, and let $A = \cup_{s \in S}A_s$. A decision maker observes the process over \(N\) epochs, \(N \geq 2\). At each \(t = 1, ..., N\), they receive a vector reward \(R_t(s, a) \in \mathbb{R}^{m}\), $m \geq 1$, for selecting action \(a\) in state \(s\); suppose that reward components are continuous on \(A_s\) for all \(s \in S\). Call \(p_t(j | s, a)\) the probability that the system will occupy state \(j\) at epoch \(t+1\) given that action \(a\) was selected in \(s\) at epoch \(t\); suppose it is continuous on \(A_s\). A Markovian decision rule \(d_{t} \in F(S, A)\) dictates the action that should be taken in each state at epoch \(t\). The set of decision rules, $D$, is assumed to be compact. No decision is taken at epoch \(N\), but a terminal state-dependent reward \(R_N(s)\) is generated. A policy specifies the decision rule that should be used at each epoch, and will be identified with its corresponding sequence of decision rules \((d_1, ..., d_{N-1})\). Let $P_W$ denote the set of all policies. 

For any policy \(\pi = (d_1, ..., d_{N-1}) \in P_W\) and any \(t < N\), let \begin{equation}u_t^{\pi}(s) =  \mathbb{E}_{\pi}^{s}[\sum_{i = t}^{N-1}R_{i}(X_i, d_i(X_i)) + R_N(X_N)]\end{equation} be the expected total reward for using \(\pi\) at epochs \(t, t+1, ..., N-1\) if the state at epoch \(t\) is \(s \in S\), where $X_i$ denotes the (random) state at epoch $i$. Since $X_t = s$, it follows from basic probability operations that
\begin{equation}
u_t^{\pi}(s) = R_t(s, d_t(s)) + \sum_{j \in S}p_t(j | s, d_t(s))u_{t+1}^{\pi}(j)
\end{equation}
where, in a way consistent with the definition above,
\begin{equation}
u_N^{\pi}(s) = R_N(s)
\end{equation}
For $t \leq N$, $s \in S$, write $V_t(s) = \bigcup_{\pi \in P_W} \{ u_t^{\pi}(s) \}$. The broad expression ``policy return", where time and state are omitted for brevity, shall refer to any member of $V_t(s)$ for some $t \leq N$ and $s \in S$. 

Define the (Pareto) efficient subset of any set \(X \subseteq \mathbb{R}^{m}\) as

\begin{equation}
e(X) = \{\, x \in X \mid \forall y \in X,\, y \geq x \implies y = x  \,\}
\end{equation}
where $\geq$ denotes the componentwise order on $\mathbb{R}^m$, i.e\ $x \geq y \iff \forall i \in \llbracket 1, m \rrbracket,\, x_i \geq y_i$ for all $x, y \in \mathbb{R}^m$. Elements of $e(X)$ are called efficient, or Pareto optimal, or admissible, or -- in set-theoretic language -- maximal \citep{geoffrion1968proper}.

With this as a background, the ``theorem" states

\begin{theorem-non}\citep{white1982}
For all \(t \leq N\) and \(s \in S\), \(e(V_t(s))\) is the unique solution \(U_t(s)\) to either of the following equations:

\begin{equation}
U_t(s) = e\Biggl(\bigcup_{a \in A_s}\biggl(\{R_t(s, a)\} \bigoplus \sum_{j \in S}p_t(j | s, a)U_{t + 1}(j)\biggr)\Biggr); \begin{array}{l}
t < N\\
\end{array}
\end{equation}
\begin{equation}
U_t(s) = \{R_N(s)\}; \begin{array}{l}
t = N\\
\end{array}
\end{equation}
where for any nonempty sets $A$ and $B$, $A \bigoplus B = \{\, a + b \mid a \in A, b \in B\,\}$. 
\end{theorem-non}

When \(m = 1\), (4) and (5) reduce to the Bellman equations referred to earlier, and the ``theorem" reduces to the correct claim that \(\max_{\pi \in P_W}u_t^{\pi}(s)\), which exists under the aforementioned assumptions \citep[Proposition 4.4.3.]{puterman1994}, is the unique solution for all \(t \leq N\) and \(s \in S\). When $m > 1$, the ``theorem" no longer holds, as the next section shows.
\ \\
\section{A Counterexample}

Consider the stationary vector-valued Markov decision process defined by $N = 4;\,
S = \{s_1, s_2\};\, A_{s_1} = A_{s_2} = A = \{a_1, a_2\};\, p_t(s_1 | s_1, a_1) = \frac{3}{4};\, p_t(s_1 | s_1, a_2) = p_t(s_1 | s_2, a_1) = p_t(s_1 | s_2, a_2) = \frac{1}{2};\, R_t(s_1, a_1) = (11, -5);\, R_t(s_1, a_2) = (9, 5);\, R_t(s_2, a_1) = (5, 5);\, R_t(s_2, a_2) = (5, -10);\, R_N(s_1) = (1, 0);\, R_N(s_2) = (0, 1)$.

There are four decision rules in this model, namely the one that prescribes $a_1$ in $s_1$ and $a_2$ in $s_2$, the one that prescribes $a_2$ in $s_1$ and $a_1$ in $s_2$, and the two that prescribe the same action in both states. This gives rise to a total of $4^{N-1} = 64$ policies. A full search of $V_1(s_1)$ based on (1) and (2) yielded \begin{equation}
e(V_1(s_1)) = \{(26.5, 5.5), (23.5, 15.5), (28.7, -2.0), (25.0, 10.5), (27.6, 0.4), (30.3, -9.0)\}
\end{equation}
whereas successive evaluation of (4) and (5) for all $t \leq N$ produced, at $t = 1$ and $s = s_1$, \begin{multline}
U_1(s_1) = \{(30.3, -9.0), (30.0, -7.8), (29.0, -3.3), (28.7, -2), (27.6, 0.4), (27.3, 1.6), (26.8, 4.2),\\ (26.5, 5.5), (25.8, 6.6), (25.6, 7.9), (25.0, 10.5), (24.1, 12.9), (23.5, 15.5)\}
\end{multline}
Clearly, $U_1(s_1) \neq e(V_1(s_1))$. This concludes the counterexample.

The problem is that for any $t < N$ and $s \in S$, $U_t(s)$ may contain infeasible vectors, that is, total rewards that are unattainable for all policies over the period spanning $t, t+1, ..., N$ given initial state $s$. Put differently, for some $w \in U_t(s)$, we may have $w \notin V_t(s)$. Indeed, in the example above, $(30.0, -7.8) \in U_1(s_1)$, yet a full search revealed no policy $\pi$ such that $u_1^{\pi}(s_1) = (30.0, -7.8)$. If $U_t(s) \not \subseteq V_t(s)$, one obviously cannot have $U_t(s) = e(V_t(s))$, and the ``theorem" fails.     

At the root of the problem is the possibility that, for some $s \in S$ and $t < N$, \begin{equation}F_t(s) := \bigcup_{a \in A_s}\biggl(\{R_t(s, a)\} \bigoplus \sum_{j \in S}p_t(j | s, a)U_{t + 1}(j)\biggr) \not \subseteq V_t(s)\end{equation}
On the page of \citep{white1982} devoted to proving the ``theorem", page 7, we find that Professor White, proceeding on the induction hypothesis that $U_{t+1}(j) = e(V_{t+1}(j))$ for all $j \in S$, for some $t < N$, merely asserts that if $w = R_t(s, a) + \sum_{j \in S}p_t(j | s, a)v_j \in F_t(s)$, with $v_j \in U_{t+1}(j)$, $a \in A_s$, we must have $w \in V_t(s)$. This, of course, need not be true, as we have just illustrated with our counterexample with $w = (30.0, -7.8)$. 

It is therefore important to inquire into the following property, which we call (P): \begin{equation}\forall s \in S,\ \forall t < N,\ F_t(s) \subseteq V_t(s)\end{equation}
Property (P) not only ensures $U_t(s) \subseteq V_t(s)$, but also that White's equations are valid, i.e\ $U_t(s) = e(V_t(s))$ for all $s \in S$ and $t \leq N$. Because the demonstration of this fact is quite lengthy, we defer it to the next section. 

It should be emphasized from the outset that (P) can only be violated when $m > 1$. Recall our earlier comment that when $m = 1$, White's equations coincide with Bellman's, since \begin{equation}e(F_t(s)) = \{\max(F_t(s))\}
\end{equation} for all $t < N$ and $s \in S$, so that we may write $U_t(s) = \{u_t(s)\} \subseteq \mathbb{R}$ for all $s \in S$ and $t \leq N$, where \begin{equation}u_t(s) = \max_{a \in A_s}\ [R_t(s, a) + \sum_{j \in S}p_t(j | s, a)u_{t+1}(j)]\end{equation} for all $s \in S$ if $t < N$, and \begin{equation}u_N(s) = R_N(s)\end{equation} for all $s \in S$. In fact, for each $s \in S$, $t < N$ and $a \in A_s$, one can demonstrate the existence of a policy $\pi(s, a, t) \in P_W$ such that \begin{equation}R_t(s, a) + \sum_{j \in S}p_t(j | s, a)u_{t+1}(j) = u_t^{\pi(s, a, t)}(s),\end{equation}
thus ensuring $F_t(s) \subseteq V_t(s)$ and therefore (P). Such a $\pi(s, a, t)$ can be constructed in two steps. First, construct a policy $\pi^{*} = (d_1^{*}, ..., d_{N-1}^{*}) \in P_W$ as follows: for each $t < N$, starting at $t = N-1$ then decreasing $t$, evaluate Equation (12) for each $s \in S$, then choose a maximizing action $a^{*} \in A_s$ and let $d_t^{*}(s) = a^{*}$. This policy satisfies $u_t(s) = u_t^{\pi}(s)$ for all $s \in S$ and $t \leq N$. Second, for each $s \in S$, $t < N$ and $a \in A_s$, choose $\pi(s, a, t)$ to be any policy that prescribes $a$ in $s$ at time $t$ and uses $\pi^{*}$ from $t+1$ onward, i.e any policy $(d_1, ..., d_{N-1})$ of the form \begin{equation}d_t(s) = a;\ d_{k}(j) = d_{k}^{*}(j), \forall k \in \llbracket t+1, N-1 \rrbracket,\, \forall j \in S\,\end{equation}
Clearly, for any such policy, we have $u_t^{\pi(s, a, t)}(s) = R_t(s, a) + \sum_{j \in S}p_t(j | s, a)u_{t+1}(j)$, hence property (P), hence the validity of the Bellman equations.

We have demonstrated, then, that (P) is a property of scalar Markov decision processes, whereas it need not be a property of vector-valued Markov decision processes with $m > 1$. In the sequel, we shall develop two overlapping classes of models which possess (P) as their property. These are:
\begin{enumerate}
\item Deterministic dynamic programs, i.e vector-valued Markov decision processes where actions, conditional on the present state, determine the next state with certainty.
\item Vector-valued Markov decision processes where the definition of ``decision rule" has been refined to include a broader range of rules than that considered by White. 
\end{enumerate}

Before treating these cases (Sections 4 and 5, respectively), we shall first justify our claim that (P) is sufficient to ensure the validity of the White equations (Section 3). We shall adopt the same method of proof as White on page 7 of \citep{white1982}, namely induction on $t$ and appeal to a lemma that will be introduced in due course. Note that since we shall be dealing with a different notion of policy in our treatment of the second case, (P) shall have to be recast accordingly, though this is a minor change, and the property's tenor as well as implication for the White equations shall remain the same. All assumptions made in the succeeding sections supplement those set forth in the Introduction, unless it is stated explicitly that a new assumption repeals an old one. Finally, it should by now be clear that in carrying out this investigation, we have an $m > 1$ in mind, though all the results presented in the following also hold when $m = 1$. 
\ \\
\section{On (P) as a Sufficient Condition for the Validity of the White Equations}
As a prelude to showing that (P) implies $U_t(s) = e(V_t(s))$ for all $s \in S$ and $t \leq N$, where $U_t(s)$ is given by Equations (4) and (5), we borrow the following lemma from \cite{white1982}:

\begin{lemma}
\citep[Lemma 2]{white1982} Let $t = 1, ..., N$ and $s \in S$. For each $u \in V_t(s)$, there is a $v \in e(V_t(s))$ such that $v \geq u$.
\end{lemma}
This lemma generalizes the fact that for a partially ordered set $(X, \succeq)$ that admits a maximum, we have $\max(X) \succeq x$ for all $x \in X$. Actually, if $V_t(s)$ does admit a maximum, such as is the case when $m = 1$, the lemma says precisely that $\max(V_t(s)) \geq u$ for all $u \in V_t(s)$. An equivalent statement is ``for all $t \leq N$, $s \in S$, $e(V_t(s))$ is the \textit{kernel} of $V_t(s)$ with respect to $\geq$", where ``kernel" is the decision-theoretic term for the unique antichain\textsuperscript{2}\footnote{\textsuperscript{2} An antichain is a subset of a partially ordered set in which no elements are comparable. Formally, $K \subseteq (X, \succeq)$ is an antichain in $X$ if for all $x, y \in K$, $x \not \succeq y$ and $y \not \succeq x$.} $K$ in a partially ordered set $(X, \succeq)$ with the property that for all $x \in X$, there exists $y \in K$ with $y \succeq x$ \citep{white1977kernels}.  

The next lemma will also be useful in proving our result. Because it can be inferred directly from the definition of $\geq$, we state it without proof:

\begin{lemma}
Let $p_1, ..., p_n$ be $n \geq 1$ nonnegative reals and $x_1, ..., x_n, y_1, ..., y_n \in \mathbb{R}^m$. Let $x, y, z$ be any vectors in $\mathbb{R}^m$. Then
\begin{enumerate}
\item If $x_i \geq y_i$ for all $i = 1, ..., n$, we have $\sum_{i = 1}^{n}p_ix_i \geq \sum_{i = 1}^{n}p_iy_i$;
\item $x \geq y$ implies $z + x \geq z + y$.
\end{enumerate}
\end{lemma}

We may now state our result.
\begin{proposition}
(P) implies $U_t(s) = e(V_t(s))$ for all $s \in S$ and $t \leq N$. 
\end{proposition}
\begin{proof}
Suppose (P) is true. Proceed by induction on $t$. For $t = N$, we have, independently of (P), $V_N(s) = \{R_N(s)\} = U_N(s)$, and thus $U_N(s) = e(V_N(s))$ for all $s \in S$. Suppose $U_{t+1}(s) = e(V_{t+1}(s))$ for all $s \in S$, for some $t < N$. 

Let $s \in S$ and $v \in e(V_t(s))$. We shall prove that $v \in U_t(s)$. Since $v \in V_t(s)$, there exists a policy $\pi = (d_1, ..., d_{N-1})$ such that $v = R_t(s, d_t(s)) + \sum_{j \in S}p_t(j | s, d_t(s))u_{t+1}^{\pi}(j)$. By Lemma 1, there is, for all $j \in S$, a policy $\pi_j$ such that $u_{t+1}^{\pi_j}(j) \geq u_{t+1}^{\pi}(j)$ and $u_{t+1}^{\pi_j}(j) \in e(V_{t+1}(j))$. Let $w = R_t(s, d_t(s)) + \sum_{j \in S}p_t(j | s, d_t(s))u_{t+1}^{\pi_j}(j)$. By the induction hypothesis, $u_{t+1}^{\pi_j}(j) \in U_{t+1}(j)$ for all $j \in S$, hence $w \in F_t(s)$. Furthermore, $w \geq v$ by Lemma 2, and $w \in V_t(s)$ by (P). But $v \in e(V_t(s))$, thus $v = w$. Therefore, $v \in F_t(s)$. Ergo, because $v \in e(V_t(s))$ and $F_t(s) \subseteq V_t(s)$ (P), we have $v \in e(F_t(s)) = U_t(s)$. This proves that $e(V_t(s)) \subseteq U_t(s)$.

To show the converse inclusion, let $u \in U_t(s)$. We may write $u = R_t(s, a) + \sum_{j \in S}p_t(j | s, a)v_j$ for some $a \in A$ and $v_j \in U_{t+1}(j)$ for all $j \in S$. From (P), $u \in V_t(s)$. Suppose that $u \notin e(V_t(s))$. Then there exists, applying Lemma 1, $v \in e(V_t(s)) \subseteq U_t(s)$ such that $v \geq u$ and $v \neq u$. This contradicts the fact that $u \in U_t(s)$, and shows that $u \in e(V_t(s))$. Consequently, $U_t(s) \subseteq e(V_t(s))$.

Because $s$ is arbitrary, we have established $U_t(s) = e(V_t(s))$ for any $s \in S$ and $t = 1, ..., N$.
\end{proof}

Before examining the case of deterministic dynamic programs, we argue that (P) is necessarily true when $N = 2$ or $3$.

\begin{proposition}
Suppose $N = 2$ or 3. Then (P) is true.
\end{proposition}

\begin{proof}
Suppose $N = 2$. For $t = 1$, we have that for all $s \in S$, for each $v \in F_1(s)$, $v = R_1(s, a) + \sum_{j \in S}p_1(j | s, a)R_2(j)$ for some $a \in A_s$. Thus $v = u_1^{\pi}(s) \in V_1(s)$ for any policy $\pi \in P_W$ that uses $a$ in $s$ at time 1. Therefore, (P) is true.

Suppose $N = 3$. For $t = 2$, we have that for all $s \in S$, for each $v \in F_2(s)$, $v = R_2(s, a) + \sum_{j \in S}p_2(j | s, a)R_3(j)$ for some $a \in A_s$. Thus $v = u_2^{\pi}(s)$ for any policy $\pi \in P_W$ that prescribes $a$ in $s$ at time 2. For this reason $F_2(s) \subseteq V_2(s)$, and therefore $e(F_2(s)) = U_2(s) \subseteq V_2(s)$, for all $s \in S$.

For $t = 1$, for each $s \in S$ and $v \in F_1(s)$, there exists an $a \in A_s$ such that $v = R_1(s, a) + \sum_{j \in S}p_1(j | s, a)v_j$ for some $v_j \in U_2(j) \subseteq V_2(j)$ for all $j \in S$. For each $j \in S$, there exists a policy $\pi_j \in P_W$ such that $v_j = u_2^{\pi_j}(j)$. Then $v = R_1(s, a) +  \sum_{j \in S}p_1(j | s, a)u_2^{\pi_j}(j) = u_1^{\pi}(s)$ for any policy $\pi \in P_W$ which prescribes $a$ in $s$ at time 1, then uses $\pi_j$ if the state at time 2 is $j$. Thus $F_1(s) \subseteq V_1(s)$ for all $s \in S$.  
\end{proof}

We now know, by Proposition 2, that the White equations are valid when $N = 2$ or $3$.

\begin{corollary}
If $N = 2$ or $3$, then $e(U_t(s)) = V_t(s)$ for all $t \leq N$ and $s \in S$.
\end{corollary}

\textit{In what follows, we assume that $N > 2$.} This makes subsequent inductive proofs less tedious by eliminating consideration of $N = 2$. 
\ \\
\section{Deterministic Dynamic Programs}

Having established that property (P) ensures the validity of the White equations, we may now use it to show that these equations are applicable to vector-valued Markov decision processes with deterministic dynamics. By ``deterministic dynamics" we mean specifically that

\begin{assumption}
For each $s \in S$, $a \in A_s$ and $t = 1, ..., N-1$, there exists $s^{+} \in S$ such that $p_t(s^{+} | s, a) = 1$. 
\end{assumption}

We claim that (P) is true under Assumption 1.

\begin{theorem}
Under Assumption 1, (P) is true, that is, \begin{equation*}\forall s \in S,\ \forall t < N,\ F_t(s) \subseteq V_t(s)\end{equation*}
\end{theorem}

\begin{proof}
We proceed by induction on $t$. For $t = N-1$, for any $s \in S$ and $a \in A_s$, by letting $f = R_{N-1}(s, a) + \sum_{j \in S}p_{N-1}(j | s, a)v_j$ where $v_j \in U_{N}(j) = \{R_N(j)\}$ for all $j \in S$, we have that $f = R_{N-1}(s, a) + v_{s^{+}} = R_{N-1}(s, a) + R_N(s^{+})$ where $s^{+} \in S$ is the state such that $p_{N-1}(s^{+} | s, a) = 1$, and thus $f = u_{N-1}^{\pi}(s)$ for any policy $\pi$ that prescribes $a$ for $s$ at time $N-1$. As a result, $F_{N-1}(s) \subseteq V_{N-1}(s)$ for all $s \in S$.

Suppose that for all $s \in S$, $F_{t+1}(s) \subseteq V_{t+1}(s)$ for some $t < N - 1$. Let $s \in S$, and let $f = R_t(s, a) + \sum_{j \in S}p_t(j | s, a)v_j$ where $v_j \in U_{t+1}(j) = e(F_{t+1}(j))$ for all $j \in S$. By Assumption 1, we may write $f = R_t(s, a) + v_{s^{+}}$, where $s^{+} \in S$ satisfies $p_{t}(s^{+} | s, a) = 1$. By the induction hypothesis, $v_{s^{+}} \in V_{t+1}(s)$, hence there exists a policy $\pi$ such that $f = R_t(s, a) + u_{t+1}^{\pi}(j)$. Let $\pi'$ be any policy that prescribes $a$ in $s$ at time $t$ then uses $\pi$ from $t+1$ onward. We have $f = u_{t}^{\pi'}(s)$, so $f \in V_t(s)$. Consequently, $F_t(s) \subseteq V_t(s)$ for all $s \in S$.

In summary, for all $t < N$ and $s \in S$, $F_t(s) \subseteq V_t(s)$.
\end{proof}

From Proposition 1 follows the validity of the White equations under Assumption 1.

\begin{corollary}
Under Assumption 1, the White equations are valid, i.e\ $U_t(s) = e(V_t(s))$ for all $t \leq N$ and $s \in S$.
\end{corollary}
\ \\
\section{The White Equations and Non-Markovian Policies}

Before turning to another class of models which have (P) as their property, we start with some useful background. Recall that the motivation for introducing (P) was the observation that, for some $t < N$ and $s \in S$, vectors in $F_t(s)$ may have no corresponding policies in $P_W$. This allows for the possibility that $F_t(s) \not \subseteq V_t(s)$, and thus for the possibility that $U_t(s) = e(F_t(s)) \not \subseteq V_t(s)$. Suppose now that the ``theorem" was true under some special set of assumptions, and that we were to show this by proving (P). Furthermore, suppose we were to proceed by induction, having noticed that the base case ($t = N-1$) holds, because for any $f = R_{N-1}(s, a) + \sum_{j \in S}p_{N-1}(j | s, a)R_N(j) \in F_{N-1}(s)$, we have $f = u_{N-1}^{\pi}(s) \in V_{N-1}(s)$ where $\pi$ is any policy that prescribes $a$ in $s$ at time $N-1$.

To carry out the induction, assume that for some $t < N-1$, \begin{equation}\forall s \in S,\ F_{t+1}(s) \subseteq V_{t+1}(s)\end{equation}
Our task is to show that for all $s \in S$, for any $w \in F_t(s)$, $w \in V_t(s)$. Let $s \in S$ and $w \in F_t(s)$. By definition of $F_t(s)$, there is an $a \in A_s$ and, by the induction hypothesis, $|S|$ policies $\pi_1, ..., \pi_{|S|}$ such that \begin{equation}w = R_t(s, a) + \sum_{j \in S}p_t(j | s, a)u_{t+1}^{\pi_j}(j)\end{equation}
In principle, a policy of the form ``if the state at time $t$ is $s$, take action $a$; if the state at time $t+1$ is $j$, take the action prescribed by $\pi_j$, and continue using $\pi_j$ over $t+2, ..., N-1$" should accrue an expected total reward of $w$ over $t, t+1, ..., N$ assuming the state at time $t$ was $s$. Such a policy, however, cannot be formulated within the present framework. Indeed, our definition of policy (more rigorously, of decision rule) presupposes that the only information relevant to decision making at time $t$ is the state at that time, whereas if we were to implement the policy just described, we would also need to know, at all times $k > t+1$, the state that was observed at time $t+1$. This leads naturally to a concept of policy where decision rules are functions of histories rather than of states, where we define ``history at time $t$" as the random trajectory of states observed prior to and including $t$, i.e \begin{equation}Z_t = (X_1, X_2, ..., X_t)\end{equation}
if $t > 1$, and $Z_1 = X_1$. As before, $X_i$ denotes the state of the process at epoch $i$.
For all $t \leq N$, let $H_t = S^{t}$ denote the set of all histories at time $t$, and notice that $H_1 = S$, so that we may use histories and states interchangeably at $t = 1$. For each epoch $t < N$, the phrase ``$t$-decision rule" shall refer to any mapping from $H_t$ to $A$. Let $D_t$ be the set of $t$-decision rules for any $t < N$, and $D = \cup_{t < N}D_t$ be the overall set of decision rules. Assume that each $D_t$ is compact. A policy $\pi = (d_1, ..., d_{N-1})$ is a sequence of decision rules where each $d_t$, $t < N$, is a $t$-decision rule. Let $P$ denote the set of all policies.

It should be made clear here that $t$-decision rules are distinct from the decision rules which \cite{puterman1994}, for example, calls history-dependent (HD in short; ``D" for ``deterministic"). For Puterman, histories contain past actions, so that an HD decision rule at epoch $t < N$ maps $S \times A \times S \times ... \times A \times S$ ($S$ $t$ times, $A$ ($t-1$) times) to $A$, while a $t$-decision rule maps $S^t$ to $A$. It is possible, however, to construe each $D_t$ as the subset of HD decision rules at epoch $t$ which ignore actions taken up to $t$. Similarly, decision rules in \cite{white1982} can be viewed as the subset of $D$ for which past states are irrelevant; that is, if $h_t = (s_1, ..., s_t)$ is any history at time $t$, and $d_t$ is a $t$-decision rule such that $d_t(h_t)$ depends on $h_t$ only through $s_t$, then $d_t$ is in fact a decision rule in White's sense. It follows that $P_W \subset P \subset P_{HD}$, where $P_{HD}$ is the set of all policies containing HD decision rules.

\begin{proposition}
$P_W \subset P$.
\end{proposition}

Now that we have refined our concepts of decision rule and policy, we may define the expected total reward accrued for using a policy $\pi = (d_1, ..., d_{N-1}) \in P$ from time $t < N$ onward if the history at $t$ is $h_t = (s_1, ..., s_t)$ as \begin{equation}u_t^{\pi}(h_t) := \mathbb{E}_{\pi}^{h_t}[\sum_{i = t}^{N-1}R_{i}(X_i, d_i(Z_i)) + R_N(X_N)].\end{equation} Since $Z_t = h_t$, we obtain a recursion analogous to Equation (1): \begin{equation}
u_t^{\pi}(h_t) = R_t(s_t, d_t(h_t)) + \sum_{j \in S}p_t(j | s_t, d_t(h_t))u_{t+1}^{\pi}(h_t, j)
\end{equation}
where we let $(h_t, j) := (s_1, ..., s_t, j)$ for all $j \in S$, and
\begin{equation}
u_N^{\pi}(h_N) = R_N(s)
\end{equation}
for all histories at time $N$, $h_N = (s_1, ..., s_N)$, with $s_N = s$. Let $V'_t(h_t) = \bigcup_{\pi \in P} \{ u_t^{\pi}(h_t) \}$ for all $t \leq N$.

Given this new framework, we are now in a position to construct policies of the form described earlier in this section, that is, policies which prescribe some action at time $t$, then, depending on the resulting state $j$ at $t+1$, pursue from that moment on a policy associated with $j$. As mentioned earlier, the importance of such policies is their ability to achieve expected total rewards of the form \begin{equation}w = R_t(s_t, a) + \sum_{j \in S}p_{t}(j | s_{t}, a)u_{t+1}^{\pi_j}(h_t, j)\end{equation}
where $\pi_1, ..., \pi_{|S|} \in P$ are arbitrary policies, $t < N$ is some decision epoch, $h_t = (s_1, ..., s_t) \in H_t$, and $a \in A_{s_t}$ some action (see Proposition 4 below). This enables us to prove a particular restatement of property (P), namely \begin{equation}\forall t < N,\ \forall h_t \in H_t,\ F'_t(h_t) \subseteq V'_t(h_t)\end{equation}
where for all $t < N$, for all $h_t = (s_1, ..., s_t) \in H_t$, 
\begin{equation}F'_t(h_t) = \bigcup_{a \in A_{s_t}}\biggl(\{R_t(s_t, a)\} \bigoplus \sum_{j \in S}p_t(j | s_t, a)U'_{t + 1}(h_t, j)\biggr)
\end{equation}
and \begin{equation}\forall t < N,\ U'_t(h_t) := e(F'_t(h_t))\end{equation} with the boundary condition \begin{equation}U'_N(h_N) := \{R_N(s)\}\end{equation} for all $h_N = (s_1, ..., s_N)$ with $s_N = s$. It is clear that this revised property, which we shall refer to as (P'), is fundamentally no different from (P); informally, we may say that it \textit{is} (P), but enunciated in a different framework. It shall later be proven that (P') provides a sufficient condition for the White equations to admit the $e(V'_t(h_t))$, $t \leq N$, $h_t \in H_t$, as their solutions, which means it plays the same role as (P) in the original framework (Propositions 5 and 6). More importantly, we shall see that (P') is necessarily true (Theorem 3), and therefore that the solutions to the White equations are the $e(V'_t(h_t))$, $t \leq N$, $h_t \in H_t$ (Theorem 5). 

\begin{proposition}
Let $t < N$ and $h_t = (s_1, ..., s_t) \in H_t$. Let $\pi_1, ..., \pi_{|S|} \in P$ be policies, and $a \in A_{s_t}$ be an action. For each policy $\pi_j$, write $\pi_j = (d^j_1, ..., d^j_{N-1})$. Let \begin{equation*}w = R_t(s_t, a) + \sum_{j \in S}p_{t}(j | s_{t}, a)u_{t+1}^{\pi_j}(h_t, j)\end{equation*}
Let
\begin{enumerate}
\item $d_t: H_t \rightarrow A$ be any $t$-decision rule such that $d_t(h_t) = a$;
\item $d_{t+1}: H_{t+1} \rightarrow A$ be any $(t+1)$-decision rule such that for all $j \in S$, $d_{t+1}(h_t, j) = d^j_{t+1}(h_t, j)$;
\item for each epoch $N-1 \geq k > t+1$, $d_k: H_k \rightarrow A$ be any $k$-decision rule such that for any $(k-t)$ states $j_{t+1}, ..., j_k \in S$, \begin{equation*}d_{k}(h_t, j_{t+1}, ..., j_k) = d^{j_{t+1}}_{k}(h_t, j_{t+1}, ..., j_k) \end{equation*}
\end{enumerate}

Finally, let $\pi$ be any element in $D_1 \times ... \times D_{N-1}$ with $(\pi)_i = d_i$ for all $i = t, t+1, ..., N-1$.

Then:
\begin{itemize}
\item[\textbf{(a)}] $\pi$ is a policy, i.e $\pi \in P$;
\item[\textbf{(b)}] $u_t^{\pi}(h_t) = w$.
\end{itemize}
\end{proposition}

\begin{proof}
Part (a) follows from the definition of $P = D_1 \times ... \times D_{N-1}$ and the well-definedness of the $d_i$'s. We shall prove part (b).

\begin{flalign*}
u_{t}^{\pi}(h_t) &= R_t(s, d_t(h_t)) + \sum_{j \in S}p_t(j | s, d_t(h_t))u_{t+1}^{\pi}(h_t, j)\\
&= R_t(s, a) + \sum_{j \in S}p_t(j | s, a)\biggl(R_{t+1}(j, d_{t+1}^{j}(h_t, j)) + \sum_{j' \in S}p_t(j' | j, d_{t+1}^{j}(h_t, j))u_{t+2}^{\pi}(h_t, j, j')\biggr)\\
&= R_t(s, a) + \sum_{j \in S}p_t(j | s, a)\biggl(R_{t+1}(j, d_{t+1}^{j}(h_t, j)) + \sum_{j' \in S}p_t(j' | j, d_{t+1}^{j}(h_t, j))u_{t+2}^{\pi_j}(h_t, j, j')\biggr)\\
&= R_t(s, a) + \sum_{j \in S}p_t(j | s, a)u_{t+1}^{\pi_j}(h_t, j)\\
&= w.
\end{flalign*}

\end{proof}

The same method which was used to prove that (P) yields $U_t(s) = e(V_t(s))$ for all $t \leq N$ and $s \in S$ (see Proposition 1, Sec. 3) can be applied here to show that (P') yields $U'_t(h_t) = e(V'_t(h_t))$ for all $t \leq N$ and $h_t \in H_t$. To do this, one will require the analogue of Lemma 1 for the $V'_t(h_t)$'s, i.e \begin{lemma}Let $t \leq N$ and $h_t \in H_t$. For each $u \in V'_t(h_t)$, there is a $v \in e(V'_t(h_t))$ such that $v \geq u$.\end{lemma} \noindent 

An essential component of our proof of Lemma 3 is a general theorem which appears in \cite{white1977kernels}, restated below in the vocabulary of the present paper:

\begin{theorem}\citep{white1977kernels}
Let $R$ be a relationship on a set $X$. Define the efficient subset of $X$ as $e(X) = \{x \in X \mid,\ yRx \implies y = x\}$. If $X$ is compact and $S(x) = \{y \in X \mid\, yRx\}$ is closed for all $x \in X$, then, for all $x \in X$, there exists $y \in e(X)$ such that $yRx$.
\end{theorem}

This theorem implies that we need only show that $V'_t(h_t)$ is compact and that $S(u) = \{v \in V'_t(h_t) \mid\, v \geq u\}$ is closed for all $t \leq N$, $h_t \in H_t$, $u \in V'_t(h_t)$. For all $t \leq N$, $h_t \in H_t$, $V'_t(h_t)$, being a subset of $\mathbb{R}^m$, is compact if and only if it is bounded and closed, as per the Bolzano-Weierstrass theorem. In this connection, it is necessary to introduce a topology on the $V'_t(h_t)$'s with respect to which convergence is defined. Accordingly, we equip these sets with the topology of pointwise convergence. Convergence in the $D_t$'s is also with respect to this topology. 

Our proof is by induction on $t$. In particular, to demonstrate the closedness of $V'_t(h_t)$ for some $t$ and $h_t \in H_t$, assuming that $V'_{t+1}(h_{t+1})$ is compact for all $h_{t+1} \in H_{t+1}$, we invoke an important lemma concerning the weighted sum of $k$ sequences, $k \geq 1$, where each sequence takes its values in a compact subset of $\mathbb{R}^m$:

\begin{lemma}
Equip $\mathbb{R}^m$ with the topology of pointwise convergence. Let $k \geq 1$ and let $U_1, ..., U_k$ be $k$ nonempty compact subsets of $\mathbb{R}^m$. If $(p_{1, n})_{n \geq 0}$, ..., $(p_{k, n})_{n \geq 0}$ are $k$ sequences with values in $[0, 1]$ and $(x_{1, n})_{n \geq 0}$, ..., $(x_{k, n})_{n \geq 0}$ sequences with values in $U_1, ..., U_k$ respectively, then there exists a strictly increasing map $\phi: \mathbb{N} \rightarrow \mathbb{N}$ such that
\begin{itemize}
\item[\textbf{(a)}] for all $i = 1, ..., k$, $p_{i, \phi(n)} \to p_i^{*}$ for some $p_i^{*} \in [0, 1]$;
\item[\textbf{(b)}] for all $i = 1, ..., k$, $x_{i, \phi(n)} \to x_i^{*}$ for some $x_i^{*} \in U_i$;
\item[\textbf{(c)}] and $\sum_{i = 1}^{k} p_{i, \phi(n)}x_{i, \phi(n)} \to \sum_{i = 1}^{k} p_i^{*}x_i^{*}$.
\end{itemize}
\end{lemma}

\begin{proof}
Let us first verify parts (a) and (b) for $k = 1$. Let $U_1$ be a nonempty compact subset of $\mathbb{R}^m$, $(p_{1, n})_{n \geq 0}$ be a sequence with values in $[0, 1]$ and $(x_{1, n})_{n \geq 0}$ be a sequence with values in $U_1$. Since $[0, 1]$ is compact, there exist a strictly increasing map $\alpha: \mathbb{N} \rightarrow \mathbb{N}$ and a $p_1^{*} \in [0, 1]$ such that $p_{1, \alpha(n)} \to p_1^{*}$. Consider now the subsequence of $(x_{1, n})_{n \geq 0}$ indexed by $\alpha$, $(x_{1, \alpha(n)})_{n \geq 0}$. Because this subsequence also takes values in a compact set, $U_1$, there exist a strictly increasing map $\beta: \mathbb{N} \rightarrow \mathbb{N}$ and a $x_1^{*} \in U_1$ such that $x_{1, \alpha(\beta(n))} \to x_1^{*}$. We also have that $(p_{1, \alpha(\beta(n))})_{n \geq 0}$ is a subsequence of $(p_{1, \alpha(n)})_{n \geq 0}$, and therefore $(p_{1, \alpha(\beta(n))})_{n \geq 0} \to p_1^{*}$. Let $\phi = \alpha \circ \beta$, which is a strictly increasing map from $\mathbb{N}$ to $\mathbb{N}$. Then $p_{1, \phi(n))} \to p_1^{*}$ and $x_{1, \phi(n)} \to x_1^{*}$, hence parts (a) and (b), from which part (c) follows.

Let now $k \geq 1$, and assume that the lemma holds for any $k$ nonempty compact subsets of $\mathbb{R}^m$. Let $U_1, ..., U_{k+1}$ be $(k+1)$ nonempty compact subsets of $\mathbb{R}^m$, $(p_{1, n})_{n \geq 0}$, ..., $(p_{k+1, n})_{n \geq 0}$ be sequences with values in $[0, 1]$, and $(x_{1, n})_{n \geq 0}$, ..., $(x_{k+1, n})_{n \geq 0}$ be sequences with values in $U_1, ..., U_{k+1}$ respectively. By the induction hypothesis, there exist $p_1^{*}$, ..., $p_k^{*}$ in $[0, 1]$, $x_1^{*}$, ..., $x_k^{*}$ in $U_1, ..., U_k$ respectively, and a strictly increasing map $\gamma: \mathbb{N} \rightarrow \mathbb{N}$ such that $p_{i, \gamma(n)} \to p_i^{*}$ and $x_{i, \gamma(n)} \to x_i^{*}$ for every $i = 1, ..., k$. Focus now on $(p_{k+1, \gamma(n)})_{n \geq 0}$. Due to the compactness of $[0, 1]$, $(p_{k+1, \gamma(n)})_{n \geq 0}$ admits a subsequence, say $(p_{k+1, \gamma(\alpha(n)}))_{n \geq 0}$, where $\alpha: \mathbb{N} \rightarrow \mathbb{N}$ is strictly increasing, such that $p_{k+1, \gamma(\alpha(n))} \to p_{k+1}^{*}$ for some $p_{k+1}^{*} \in [0, 1]$. Also, we know that $p_{i, \gamma(\alpha(n))} \to p_i^{*}$ and $x_{i, \gamma(\alpha(n))} \to x_i^{*}$ for every $i = 1, ..., k$. By letting $\phi = \gamma \circ \alpha$, we can see that for every $i = 1, ..., k+1$, $p_{i, \phi(n)} \to p_i^{*}$ and $x_{i, \phi(n)} \to x_i^{*}$. Thus $\sum_{i = 1}^{k+1} p_{i, \phi(n)}x_{i, \phi(n)} \to \sum_{i = 1}^{k+1} p_i^{*}x_i^{*}$. The lemma follows by induction.
\end{proof}

Theorem 2 and Lemma 4 supply enough background material for proving Lemma 3. We shall now proceed to do so.

\begin{proof}[Proof of Lemma 3]
For notational convenience, let $S = \{j_1, ..., j_{|S|}\}$. We begin by demonstrating the compactness of $V'_t(h_t)$ for all $t \leq N$ and $h_t \in H_t$. The proof is by induction on $t$. For $t = N$, let $h_N = (s_1, ..., s_N) \in H_N$. $V'_N(h_N) = \{R_N(s_N)\}$ is closed because it is finite. For any $\pi \in P$, we have $u_N^{\pi}(h_N) = R_N(s_N)$, and therefore $\|u_N^{\pi}(h_N)\|_{\infty} \leq \|R_N(s_N)\|_{\infty}$. We then have that for all $h_N \in H_N$, $V'_N(h_N)$ is closed and bounded, and therefore compact. Suppose $V'_{t+1}(h_{t+1})$ is compact for all $h_{t+1} \in H_{t+1}$, for some $t < N$.

Let $h_t = (s_1, ..., s_t) \in H_t$ and $v = u_{t}^{\pi}(h_{t}) \in V'_t(h_t)$ with $P = (d_1, ..., d_{N-1})$. Then $v = R_t(s_t, d_t(h_t)) + \sum_{j \in S}p_t(j | s_t, d_t(h_t))u_{t+1}^{\pi}(h_t, j)$. By the induction hypothesis, there exists, for each $j \in S$, an $M_j \geq 0$ such that $\|u_{t+1}^{\pi}(h_t, j)\|_{\infty} \leq M_j$. Applying the triangle inequality twice on $\|v\|_{\infty}$, we obtain $\|v\|_{\infty} \leq \|R_t(s_t, d_t(h_t))\|_{\infty} + \sum_{j \in S} M_j$, whence \begin{equation*}\|v\|_{\infty} \leq \max_{1 \leq i \leq m} \max_{a \in A_{s_t}}  |R_t(s_t, a)_i| + \sum_{j \in S} M_j\end{equation*}
\noindent by noticing that for each $1 \leq i \leq m$, $\max_{a \in A_{s_t}}\, |R_t(s_t, a)_i|$ exists because $A_{s_t}$ is compact and $a \mapsto |R_t(s_t, a)_i|$ is continuous on $A_{s_t}$.

Thus $V'_t(h_t)$ is bounded. In order to demonstrate its closedness, let $(v^n)_{n \geq 0}$ be a sequence of vectors in $V'_t(h_t)$ that converges to a $v \in \mathbb{R}^m$. We may write \begin{equation*}v^n = u_t^{\pi_n}(h_t) = R_t(s_t, d_t^{\pi_n}(h_t)) + \sum_{i = 1}^{|S|}p_t(j_i | s_t, d_t^{\pi_n}(h_t))u_{t+1}^{\pi_n}(h_t, j_i)\end{equation*} for some $\pi_n = (d_1^{\pi_n}, ..., d_{N-1}^{\pi_n}) \in P$, for all $n \geq 0$. We endeavor to show that $v \in V'_t(h_t)$. By the hypothesized compactness of the $V'_{t+1}(h_t, j_i)$'s, $i = 1, ..., |S|$, there exist, according to Lemma 4, a strictly increasing map $\phi: \mathbb{N} \rightarrow \mathbb{N}$, and $p_1^{*}$, ..., $p_{|S|}^{*} \in [0, 1]$, and $x_1^{*}, ..., x_{|S|}^{*}$ in $V'_{t+1}(h_t, j_1)$, ..., $V'_{t+1}(h_t, j_{|S|})$ respectively, such that \begin{equation*}u_{t+1}^{\pi_{\phi(n)}}(h_t, j_i) \to x_i^{*}\end{equation*}
\begin{equation*}p_t(j_i | s_t, d_t^{\pi_{\phi(n)}}(h_t)) \to p_i^{*}\end{equation*}
for all $i = 1, ..., |S|$, and therefore $\sum_{i = 1}^{|S|}p_t(j_i | s_t, d_t^{\pi_{\phi(n)}}(h_t))u_{t+1}^{\pi_{\phi(n)}}(h_t, j_i) \to \sum_{i = 1}^{|S|}p_i^{*}x_i^{*}$.

\noindent Focusing now on the $p_i^{*}$'s, observe that since $D_t$ is compact, there exist a $d^{*} \in D_t$ and a strictly increasing $\alpha: \mathbb{N} \rightarrow \mathbb{N}$ such that $d_t^{\pi_{\phi(\alpha(n))}} \to d^{*}$ and, in particular, in view of the adopted topology, $d_t^{\pi_{\phi(\alpha(n))}}(h_t) \to d^{*}(h_t)$. By the assumed continuity of transition probabilities on $A_{s_t}$, it follows that \begin{equation*}p_t(j_i | s_t, d_t^{\pi_{\phi(\alpha(n))}}(h_t)) \to p_t(j_i | s_t, d^{*}(h_t))\end{equation*}
for all $i = 1, ..., |S|$. Moreover, for all $i = 1, ..., |S|$, $(p_t(j_i | s_t, d_t^{\pi_{\phi(\alpha(n))}}(h_t))_{n \geq 0}$ is a subsequence of $(p_t(j_i | s_t, d_t^{\pi_{\phi(n)}}(h_t)))_{n \geq 0}$, from which it follows that $p_i^{*} = p_t(j_i | s_t, d^{*}(h_t))$. To recapitulate, then, we have established that \begin{equation*}\sum_{i = 1}^{|S|}p_t(j_i | s_t, d_t^{\pi_{\phi(\alpha(n))}}(h_t))u_{t+1}^{\pi_{\phi(\alpha(n))}}(h_t, j_i) \to \sum_{i = 1}^{|S|}p_t(j_i | s_t, d^{*}(h_t))x_i^{*}.\end{equation*}

\noindent By the assumed continuity of rewards on $A_{s_t}$, we also have that $R_t(s_t, d_t^{\pi_{\phi(\alpha(n))}}(h_t)) \to R_t(s_t, d^{*}(h_t))$. This means that \begin{multline*}R_t(s_t, d_t^{\pi_{\phi(\alpha(n))}}(h_t)) + \sum_{i = 1}^{|S|}p_t(j_i | s_t, d_t^{\pi_{\phi(\alpha(n))}}(h_t))u_{t+1}^{\pi_{\phi(\alpha(n))}}(h_t, j_i) \\ \to R_t(s_t, d^{*}(h_t)) + \sum_{i = 1}^{|S|}p_t(j_i | s_t, d^{*}(h_t))x_i^{*}\end{multline*}

\noindent We thus have a subsequence of $(v^n)_{n \geq 0}$ that converges to $R_t(s_t, d^{*}(h_t)) + \sum_{i = 1}^{|S|}p_t(j_i | s_t, d^{*}(h_t))x_i^{*}$. Hence $v = R_t(s_t, d^{*}(h_t)) + \sum_{i = 1}^{|S|}p_t(j_i | s_t, d^{*}(h_t))x_i^{*}$. Recalling that each $x_i^{*}$ is in $V'_t(h_t, j_i)$, we know from Proposition 4 that there is a $\pi \in P$ such that $v = u_t^{\pi}(h_t)$. Thus $v \in V'_t(h_t)$. In conclusion, $V'_t(h_t)$ is closed and bounded, and therefore compact.

Finally, for all $t \leq N$, $h_t \in H_t$ and $u \in V'_t(h_t)$, that $S(u) = \{v \in V'_t(h_t) \mid\, v \geq u\}$ is closed follows immediately from the closedness of $V'_t(h_t)$ and the fact that if a convergent sequence $(v^n)_{n \geq 0}$ in $\mathbb{R}^m$ satisfies $v^n \geq u$ for all $n \geq 0$, then its limit $v$ also satisfies $v \geq u$ by the definition of $\geq$.

By Theorem 4, Lemma 3 is true.
\end{proof}

\begin{proposition}
(P') implies $U'_t(h_t) = e(V'_t(h_t))$ for all $t \leq N$ and $h_t \in H_t$.
\end{proposition}

\begin{proof}
Suppose (P') is true. For $t = N$, for all $h_N = (s_1, ..., s_N) \in H_N$, $U'_N(h_N) = \{R_N(s_N)\}$ and $V'_N(h_N) = \{R_N(s_N)\}$; thus, $U'_N(h_N) = e(V'_N(h_N))$. Let $t < N-1$ and assume $U'_{t+1}(h_{t+1}) = e(V'_{t+1}(h_{t+1})$ for all $h_{t+1} \in H_{t+1}$.

Let $h_t = (s_1, ..., s_t) \in H_t$ and $v \in e(V'_t(h_t))$. There exists a $\pi = (d_1, ..., d_{N-1}) \in P$ such that $v = R_t(s_t, d_{t}(h_t)) + \sum_{j \in S}p_t(j | s_t, d_{t}(h_t))u_{t+1}^{\pi}(h_t, j)$. For all $j \in S$, $(h_t, j) \in H_{t+1}$, and by Lemma 3 there exists a $v_j \in e(V'_{t+1}(h_t, j))$ such that $v_j \geq u_{t+1}^{\pi}(h_t, j)$. Let $w = R_t(s_t, d_{t}(h_t)) + \sum_{j \in S}p_t(j | s_t, d_{t}(h_t))v_j$. By Lemma 2, $w \geq v$, and by the induction hypothesis together with (P') we have that $w \in V'_t(h_t)$. However, $v$ being efficient in $V'_t(h_t)$, $v = w$. Hence, again by the induction hypothesis, $v \in F'_t(h_t)$. By (P'), $F'_t(h_t) \subseteq V'_t(h_t)$, and thus $v$ is also efficient in $F'_t(h_t)$, i.e $v \in U'_t(h_t)$. From this it follows that $e(V'_t(h_t)) \subseteq U'_t(h_t)$.

Now we shall demonstrate the converse inclusion. Let $u \in U'_t(h_t)$.  There exists an $a \in A_{s_t}$ such that $u = R_t(s_t, a) + \sum_{j \in S}p_t(j | s_t, a)v_j$ for some $v_j \in U'_{t+1}(h_t, j)$ for all $j \in S$. By (P'), $u \in V'_t(h_t)$. Assuming towards a contradiction that $u \notin e(V'_t(h_t))$, there exists, applying Lemma 1, $v \in e(V'_t(h_t)) \subseteq U'_t(h_t)$ such that $v \geq u$ and $v \neq u$. This contradicts the fact that $u \in U'_t(h_t)$, because $U'_t(h_t)$ is an efficient set, and shows that $u \in e(V'_t(h_t))$. Consequently, $U'_t(h_t) \subseteq e(V'_t(h_t))$.

Finally, for all $t \leq N$, for all $h_t \in H_t$, $U'_t(h_t) = e(V'_t(h_t))$.

\end{proof}

We are now able to substantiate (P').

\begin{theorem}
(P') is true, i.e \begin{equation*}\forall t < N,\ \forall h_t \in H_t,\ F'_t(h_t) \subseteq V'_t(h_t)\end{equation*}
\end{theorem}

\begin{proof}
We proceed by induction on $t$. Let $h_{N-1} = (s_1, ..., s_{N-1}) \in H_{N-1}$, and let $f = R_{N-1}(s_{N-1}, a) + \sum_{j \in S}p_{N-1}(j | s_{N-1}, a)R_N(j) \in F'_{N-1}(h_{N-1})$, for some $a \in A_{s_{N-1}}$. Then $f = u_{N-1}^{\pi}(h_{N-1}) \in V'_{N-1}(h_{N-1})$ for any policy $\pi \in P$ that prescribes $a$ at time $N-1$ when the history at that time is $h_{N-1}$. Consequently, $F'_{N-1}(h_{N-1}) \subseteq V'_{N-1}(h_{N-1})$ for all $h_{N-1} \in H_{N-1}$.

Let us assume now that for some $t < N-1$, $F'_{t+1}(h_{t+1}) \subseteq V'_{t+1}(h_{t+1})$ for all $h_{t+1} \in H_{t+1}$. Let $h_t = (s_1, ..., s_t) \in H_t$, and let $f \in F'_t(h_t)$. There exist an $a \in A_{s_t}$ and, for each $j \in S$, a $v_j \in U'_{t+1}(h_t, j)$ such that $f = R_t(s_t, a) + \sum_{j \in S}p_{t}(j | s_{t}, a)v_j$. By our induction hypothesis, $U'_{t+1}(h_t, j) = e(F'_{t+1}(h_t, j)) \subseteq V'_{t+1}(h_t, j)$ for all $j \in S$, so that each $v_j$ can be written $v_j = u_{t+1}^{\pi_j}(h_t, j)$ for some policy $\pi_j \in P$. Thus, \begin{equation}f = R_t(s_t, a) + \sum_{j \in S}p_{t}(j | s_{t}, a)u_{t+1}^{\pi_j}(h_t, j)\end{equation} 

By virtue of Proposition 4, there exists a policy $\pi \in P$ such that $f = u_t^{\pi}(h_t)$. Thus, $f \in V'_t(h_t)$. As a result, $F'_t(h_t) \subseteq V'_t(h_t)$. In sum, we have shown by induction that for all $t < N$, $F'_t(h_t) \subseteq V'_t(h_t)$ for all $h_t \in H_t$.  
\end{proof}

Combining Theorem 3 with Proposition 5, we obtain the following relation between the solutions to Equations (23) and (24) and the efficient sets of policy returns over $P$, the $e(V'_t(h_t))$'s:

\begin{theorem}
For all \(t \leq N\) and \(h_t = (s_1, ..., s_t) \in H_t\), \(e(V'_t(h_t))\) is the unique solution \(U'_t(h_t)\) to either of the following equations:

\begin{equation*}
U'_t(h_t) = e\Biggl(\bigcup_{a \in A_{s_t}}\biggl(\{R_t(s_t, a)\} \bigoplus \sum_{j \in S}p_t(j | s_t, a)U'_{t + 1}(h_t, j)\biggr)\Biggr); \begin{array}{l}
t < N\\
\end{array}
\end{equation*}
\begin{equation*}
U'_t(h_t) = \{R_N(s_t)\}; \begin{array}{l}
t = N\\
\end{array}
\end{equation*}
\end{theorem}

The equations can be simplified by noticing that $U'_t(h_t)$ depends on $h_t$ only through $s_t$. Because these and White's equations share the same boundary condition at $t = N$, this has the important implication that for all $t \leq N$ and $h_t = (s_1, ..., s_t) \in H_t$, $U'_t(h_t)$ is none other than $U_t(s_t)$.

\begin{proposition}
For all $t \leq N$, for all $h_t = (s_1, ..., s_t) \in H_t$, $U'_t(h_t) = U_t(s_t)$. 
\end{proposition}

\begin{proof}
For $t = N$, we have that for all $h_N = (s_1, ..., s_N) \in H_N$, $U'_N(h_N) = \{R_N(s_N)\} = U_N(s_N)$. The property holds for $t = N$. Assume that for some $t < N$, $U'_{t+1}(h_{t+1}) = U_{t+1}(s_{t+1})$ for all $h_{t+1} = (s_1, ..., s_{t+1}) \in H_{t+1}$. For all $h_t = (s_1, ..., s_t) \in H_t$, we have by definition, \begin{equation}U'_t(h_t) = e\Biggl(\bigcup_{a \in A_{s_t}}\biggl(\{R_t(s_t, a)\} \bigoplus \sum_{j \in S}p_t(j | s_t, a)U'_{t + 1}(h_t, j)\biggr)\Biggr)\end{equation}
hence
\begin{equation}U'_t(h_t) = e\Biggl(\bigcup_{a \in A_{s_t}}\biggl(\{R_t(s_t, a)\} \bigoplus \sum_{j \in S}p_t(j | s_t, a)U_{t + 1}(j)\biggr)\Biggr)
\end{equation}
by the induction hypothesis. Thus $U'_t(h_t) = U_t(s_t)$. The proposition follows.
\end{proof}

Applying Proposition 6 and Theorem 4, we discover that the solutions to the White equations at $t \leq N$ are the $e(V'_t(h_t))$'s, $h_t \in H_t$.

\begin{theorem}
For all $t \leq N$, for all $h_t = (s_1, ..., s_t) \in H_t$, $U_t(s_t) = e(V'_t(h_t))$.
\end{theorem}

An interesting byproduct of Theorem 5 and Corollary 2 is that under the determinism assumption defined in Section 3, all efficient policy returns over $P$ are attained by policies in $P_W \subset P$. Hence, given the same initial state, an optimal policy in $P_W$ can accrue as ``large" an expected reward as an optimal policy in the whole of $P$. Rather than consider the full range of policies in $P$, the decision maker is therefore justified in focusing only on Markovian policies, i.e\ $P_W$, which are easier to implement and evaluate \citep{puterman1994}. 

\begin{corollary}
Suppose Assumption 1 is satisfied. Then, for all $s \in S$, \begin{equation*}e\Biggl(\bigcup_{\pi \in P_W} \{ u_1^{\pi}(s) \}\Biggr) = e\Biggl(\bigcup_{\pi \in P} \{ u_1^{\pi}(s) \}\Biggr)\end{equation*}
\end{corollary}
\ \\
\section{An Example}
In this section we report the results of experimental tests of Corollary 2 (Sec. 4), Theorem 5 (Sec. 5) and Corollary 3 (Sec. 5). To check Theorem 5, we compared $U_1(s_1)$ of the model in Section 2 with $e(V'_1(h_1)) = e(V'_1(s_1))$. Recall that this model furnished a counterexample to White's ``theorem", since \begin{multline}
U_1(s_1) = \{(30.3, -9.0), (30.0, -7.8), (29.0, -3.3), (28.7, -2), (27.6, 0.4), (27.3, 1.6), (26.8, 4.2),\\ (26.5, 5.5), (25.8, 6.6), (25.6, 7.9), (25.0, 10.5), (24.1, 12.9), (23.5, 15.5)\}
\end{multline}
yet $e(V_1(s_1)) \neq U_1(s_1)$.

A full search that was performed over the 16\,384 policies in $P$ yielded \begin{multline}
e(V'_1(s_1)) = \{(30.3, -9.0), (30.0, -7.8), (29.0, -3.3), (28.7, -2), (27.6, 0.4), (27.3, 1.6), (26.8, 4.2),\\ (26.5, 5.5), (25.8, 6.6), (25.6, 7.9), (25.0, 10.5), (24.1, 12.9), (23.5, 15.5)\}
\end{multline}
Thus $U_1(s_1) = e(V'_1(s_1))$, a fact in accord with Theorem 5.

As a test of Corollary 2, we next imposed deterministic state transitions on this model, letting $p_t(s_1 | s_1, a_1) = p_t(s_2 | s_1, a_2) = p_t(s_1 | s_2, a_1) = p_t(s_2 | s_2, a_2) = 1$ for all $t < N = 4$. The rewards were unchanged from Section 2. By evaluating Equations (4) and (5) backward in time, we obtained \begin{equation}U_1(s_1) = \{(34, -15), (26, 5), (31, -4), (23, 16)\}\end{equation}
and a full search of the 64 policies in $P_W$ confirmed $e(V_1(s_1)) = U_1(s_1)$.

To support Corollary 3, we calculated $e(V'_1(s_1))$ in this deterministic setting with a view to comparing it with $e(V_1(s_1))$. Through a comprehensive search of $P$, we found \begin{equation}e(V'_1(s_1)) = \{(34, -15), (26, 5), (31, -4), (23, 16)\} = e(V_1(s_1))\end{equation}
which corroborates Corollary 3.
\ \\
\section{Conclusion}

The aims of this paper were (a) to show that the hypotheses underlying the vector extension of the Bellman equations of a Markov decision process do not ensure the extension's validity, and (b) to propose alternative hypotheses that do. A counterexample to the ``theorem" on which this extension is predicated was provided, and an explanation as to why the ``theorem" failed was advanced. It was found that the ``theorem" holds when (1) the decision making horizon, $N$, spans two or three time steps, or when (2) only deterministic state transitions are permitted, or when (3) the policy space is enlarged from just Markovian policies to include non-Markovian ones.

Not covered in this paper was a fourth area in which the extension is valid, namely a combination of determinism (Assumption 1) and that $\geq$ be any relationship on $\mathbb{R}^m$ that meets either of the following sets of conditions:

\begin{enumerate}
\item $\geq$ is transitive;
\item for all $t \leq N$, $s \in S$, and $u \in V_t(s)$, the set $S(u) = \{v \in V_t(s) \mid,\ v \geq u\}$ is closed;
\item for all $x, y, z \in \mathbb{R}^m$, $x \geq y \implies (x+z) \geq (y+z)$;
\item for all $x \in \mathbb{R}^m$, for all nonnegative reals $a$, $x \geq y \implies ax \geq ay$;
\end{enumerate}

or

\begin{enumerate}
\item $\geq$ is a partial order;
\item for all $t \leq N$, $s \in S$, every totally ordered subset (chain) of $V_t(s)$ admits an upper bound in $V_t(s)$ with respect to $\geq$;
\item for all $x, y, z \in \mathbb{R}^m$, $x \geq y \implies (x+z) \geq (y+z)$;
\item and for all $x \in \mathbb{R}^m$, for all nonnegative reals $a$, $x \geq y \implies ax \geq ay$.
\end{enumerate}

In the remainder of the discussion, we shall refer to the first set of conditions as (C1), and to the second as (C2). Some motivation for (C1) and (C2) follows. In our proof of the proposition that (P) implies $U_t(s) = e(V_t(s))$ for all $t$ and $s$, we appealed principally to two facts concerning $\geq$: Lemma 1 and Lemma 2. Here Lemma 2 follows from conditions (3) and (4) in both of (C1) and (C2). Lemma 1 originally appeared in \cite{white1982}, and is an application of a general theorem published years earlier in \cite{white1977kernels}. The theorem states that if $R$ is a transitive relationship on a compact set $X$, and $S(x) = \{y \in X \mid,\ yRx\}$ is closed for all $x \in X$, then for all $x \in X$, there exists $y \in e(X)$ with $yRx$, where $e(X)$ is defined with respect to $R$. From this follows Lemma 1 under (C1), because the compactness of $V_t(s)$ for all $t$ and $s$ is a fact independent of $\geq$ \citep{white1982}. As to why it follows under (C2) as well, recall that by Zorn's lemma \citep{zorn1935remark}, a partially ordered set $X$ in which every chain is upper bounded in $X$ has at least a maximal element, i.e\ $e(X) \neq \emptyset$. Under (C2), it is straightforward to show that every chain in $V_t(s) \cap S(u)$ is upper bounded in $V_t(s) \cap S(u)$, for all $t$, $s$ and $u \in V_t(s)$. Thus $e(V_t(s) \cap S(u)) \neq \emptyset$. Furthermore, as can easily be checked, $e(V_t(s) \cap S(u)) \subseteq e(V_t(s)) \cap S(u)$, and thus $e(V_t(s)) \cap S(u) \neq \emptyset$, hence Lemma 1. From this we conclude that Proposition 1 holds. At this stage we need only prove that (P) is true, but this is a direct consequence of Assumption 1. Therefore, the extension is valid under Assumption 1 coupled with (C1) or (C2).
\\
\section*{Acknowledgment}
The author is indebted to Dominikus Noll for stimulating conversations that gave impetus to the section pertaining to the deterministic case. The author also thanks Eilon Solan for comments on an early version of this paper, and Slim Kammoun for assistance with certain mathematical aspects of the final section.
\\
\section*{Statements and Declarations}
The author states that he has no conflict of interest.
$\,$

$\,$

\end{document}